\documentclass{amsart}
\usepackage{amssymb, epic,eepic,epsfig,amsbsy,amsmath,amscd,url}
\usepackage{color}
\usepackage{verbatim} 

\def\be#1\ee{\begin{equation}#1\end{equation}}
\theoremstyle{plain}

\newtheorem{Thm}{Theorem}

\newtheorem{proposition}{Proposition}[section]
\newtheorem{lemma}[proposition]{Lemma}
\newtheorem*{lemma*}{Lemma}
\newtheorem*{Theorem*}{Theorem}

\theoremstyle{definition}

\newtheorem{remark}{Remark}[section]
\newtheorem*{remark*}{Remark}
\newtheorem*{example*}{Example}

\def\printname#1{
    \if\draft
        \smash{\makebox[0pt]{\hspace{-0.5in}
            \raisebox{8pt}{\tt\tiny #1}}}
    \fi
}

\newlength{\standardunitlength}
\catcode`\@=11
\long\def\@makecaption#1#2{%
     \vskip 10pt

\setbox\@tempboxa\hbox{
       \small\sf{\bfcaptionfont #1. }\ignorespaces #2}%
     \ifdim \wd\@tempboxa >\captionwidth {%
         \rightskip=\@captionmargin\leftskip=\@captionmargin
         \unhbox\@tempboxa\par}%
       \else
         \hbox to\hsize{\hfil\box\@tempboxa\hfil}%
     \fi}
\font\bfcaptionfont=cmssbx10 scaled \magstephalf
\newdimen\@captionmargin\@captionmargin=2\parindent
\newdimen\captionwidth\captionwidth=\hsize
\catcode`\@=12

\newcommand{\tr}{\operatorname{tr}}
\newcommand{\im}{\operatorname{Im}}

\newcommand{\cC}{{\mathcal C}}

\newcommand{\Tor}{{\operatorname{\mathfrak {tor}}}}

\newcommand{\beq}{\begin{equation}}
\newcommand{\eeq}{\end{equation}}
\newcommand{\bThm}{\begin{Thm}}
\newcommand{\eThm}{\end{Thm}}

\def\BZ{\mathbb Z}

\def\BQ{\mathbb Q}
\def\BR{\mathbb R}
\def\BC{\mathbb C}

\def\DD{\mathcal D}

\def\la{\langle}
\def\ra{\rangle}

\def\vol{\operatorname{vol}}

\newcommand{\bproof}{\begin{proof}}
\newcommand{\blemma}{\begin{lemma}}
\newcommand{\eproof}{\end{proof}}
\newcommand{\elemma}{\end{lemma}}

\def\im{\operatorname{Im}}

\def\rk{\operatorname{rk}}

\def\Hom{\mathrm{Hom}}

\def\al{\alpha}

\def\Z{\BZ}

\def\bk{\mathbf{k}}

\def\tr{\mathrm{tr}}

\def\bz{\mathbf z}

\newcommand\no[1]{}

\begin{document}

\title[Torsion  Growth]{Growth of Regulators in finite abelian coverings}

\author{Thang T. Q. L\^e}
\address{Department of Mathematics, Georgia Institute of Technology,
Atlanta, GA 30332--0160, USA }
\email{letu@math.gatech.edu}

{ \noindent
\thanks{The author was supported in part by National Science Foundation.} 
}
\begin{abstract}
We show that the regulator, which is the difference between the homology torsion and the combinatorial Ray-Singer torsion, of finite abelian coverings of a fixed complex
has sub-exponential growth rate.

\end{abstract}

\maketitle

\section{Introduction}
\subsection{Based free complex over group ring and its quotients} Suppose $\pi$ is a finitely presented  group and  $\BZ[\pi]$ is the group ring of $\pi$ over the ring $\BZ$ of integers.

Let $\cC$ be a finitely-generated   based
free $\BZ[\pi]$-complex
$$
0 \to C_m \overset {\partial_m}\longrightarrow C_{m-1} \overset {\partial_{m-1}}\longrightarrow C_{m-2} \dots \overset {\partial_2}\longrightarrow C_{1}\overset {\partial_1}\longrightarrow C_{0}\to 0.
$$
Here ``based free" means each $C_k$ is a free $\BZ[\pi]$-module equipped with a preferred base.

For a normal subgroup $\Gamma  \lhd \pi$ let $\cC_\Gamma := \BZ[\pi/\Gamma ] \otimes _{\BZ[\pi]} \cC$. Assume  that the  index $[\pi:\Gamma ]$ is finite. Then $\cC_\Gamma $ is a finitely-generated based
free $\BZ$-complex, where the preferred base of $ \BZ[\pi/\Gamma ] \otimes_{\BZ[\pi]} C_k$ is defined using the one of $C_k$ in a natural way.

A prototypical case is the following. Suppose $\tilde X \to X$ is a regular covering with $\pi$ the group of deck transformations
and $X$ a finite CW-complex. Choose a lift in $\tilde X$ of  each cell of $X$.
Then the CW-complex $\cC$ of $\tilde X$ induced from that of $X$ is a finitely-generated   based
 free $\BZ[\pi$-complex. For a normal subgroup $\Gamma \lhd \pi$, $\cC_\Gamma$ is
the CW-complex of the covering $X_\Gamma$, corresponding to the group $\Gamma$, and $H_k(\cC_\Gamma)$ is the $k$-th homology of the covering $X_\Gamma$. Usually  interesting invariants
do not depend on the choice of the lifts of cells of $X$.

\subsection{Two torsions}
We can define two torsions of the quotient complex $\cC_\Gamma $, the homology torsion $\tau^{H}(\cC_\Gamma )$ and the combinatorial Ray-Singer torsion $\tau^{RS}(\cC_\Gamma )$, as follows.
The homology torsion is
$$ \tau ^H(\cC_\Gamma):=\left(  {\prod_k}^* |\Tor_\BZ(H_k(\cC_\Gamma))| \right)^{-1}\in  \BR_+, $$
where $\Tor_\BZ(M)$ is the $\BZ$-torsion part of the finitely-generated abelian group $M$, and ${\prod_k}^* a_k$ is the alternating product
$$ {\prod_k}^* a_k = {\prod_k} a_k^{(-1)^k}.$$

The  Ray-Singer torsion of $\cC_\Gamma$ is
$$\tau^{RS}(\cC_\Gamma)= {\prod_k}^* {\det}' (\partial_k) \in \BR_+,$$
where $\det '$ is the geometric determinant 
of  linear maps between based Hermitian spaces. We recall the definition of $\det'$ in Section \ref{sec.volume}.

\subsection{Comparison: general question}
We want to compare the asymptotics of the two torsions as $\Gamma $ becomes ``thinner and thinner in $\pi$", so that $\pi/\Gamma$ approximates $\pi$  in the following sense.
A finite set $S$ of generators of $\pi$ defines a word length function $l_S$ (and hence a metric)  on $\pi$.  Define
$$ \la \Gamma  \ra := \min \{ l_S(x) \mid x \in \Gamma  \setminus \{e\}\}.$$
Here $e$ is the unit of $\pi$. 
In all what follow, statements do not depend on the choice of the generator set $S$, since the metrics of two different generator sets are quasi-isometric.

We are interested in the following question: Suppose $\cC$ is $L^2$-acyclic (see e.g. \cite{Luck_book}). Under what conditions does it hold that
\be
\lim_{\la \Gamma  \ra \to \infty, |\pi:\Gamma | < \infty}\frac{\ln (\tau^{H}(\cC_\Gamma )) -\ln  (\tau^{RS}(\cC_\Gamma ))   }{|\pi:\Gamma |} =0.
\label{eq.00}
\ee

The motivation of this question comes from the question \cite{Luck_book}: can one  approximate  $L^2$-torsions by finite-dimensional analogs? In some favorable conditions, one expects that the growth rate of each of $\tau^H$ and $\tau^{RS}$ is the $L^2$-torsion, hence they  must be the same.
\begin{remark} (a)
If $\{\Gamma _n, n=1,2, \dots\}$ is a sequence of exhausted nested normal subgroups of $\pi$, i.e. $\Gamma _{n+1} \subset \Gamma _n$ and $\cap_n \Gamma _n =\{e\}$, then $\lim_{n\to \infty } \la \Gamma _n \ra =\infty$.
The limit in \eqref{eq.00} is more general (stronger) than the limit of an exhausted nested sequence, as we don't have ``nested" property.

(b) There exists a sequence $\Gamma_n \lhd \pi$ such that $\lim_{n\to \infty } \la \Gamma _n \ra =\infty$ if and only if $\pi$ is residually finite. Hence, the left hand side of \eqref{eq.00} makes sense
only when $\pi$ is residually finite.

(c) Define $\tr_\pi(x) = \delta_{x,e}$ for $x\in \pi$. This functional trace is the base for the definition of many combinatorial $L^2$-invariants.
For a fixed $x \in \pi$, we have 
$$ \lim_{\la \Gamma  \ra \to \infty} \tr_{\Gamma}x = \tr_\pi(x).$$
This is the reason why one expects that as  ${\la \Gamma  \ra \to \infty}$, many $L^2$-invariants (under some technical conditions) can be approximated by the corresponding invariants of $\pi/\Gamma$.
\end{remark}

\subsection{Main results}

The main result of the paper treats the case $\pi= \BZ^n$.

\begin{Thm}
\label{thm.01}
Suppose  $\cC$ is an $L^2$-acyclic  finitely generated based free $\BC[\BZ^n]$-complex. Then \eqref{eq.00}, with $\pi=\BZ^n$, holds true.
\end{Thm}
We will not give the definition of $L^2$-acyclicity. Instead,  for $\pi=\BZ^n$,  we will use an equivalent definition \cite{Elek,Luck_book}: the $L^{(2)}$ homology $H_k^{(2)}(\cC)$ vanishes if and only if
$ H_k(\cC\otimes _{\BZ[\BZ^n]} F)=0$. Here $F$ is the fractional field of the commutative domain $\BZ[\BZ^n]$.

It is expected that \eqref{eq.00} holds for  a large class of residually non-abelian finite groups (like hyperbolic groups), but a proof is probably very difficult.
\begin{remark}
(a) Our result does not imply that 
$$ \lim_{\la \Gamma  \ra \to \infty, |\pi:\Gamma | < \infty}\frac{\ln (\tau^{H}(\cC_\Gamma ))   }{|\pi:\Gamma |}=    
\lim_{\la \Gamma  \ra \to \infty, |\pi:\Gamma | < \infty}\frac{\ln  (\tau^{RS}(\cC_\Gamma ))  }{|\pi:\Gamma |},  
$$
as we cannot prove the existence of each of the limits. As mentioned above, for a large class of (abelian and non-ablian groups) and maybe under some restrictions
on $\Gamma$, one expects that both limits exist, are the same, and equal to the $L^2$-torsion of $\cC$. This holds true for $\pi = \BZ$, with no restrictions on $\Gamma$, see \cite{GS,Riley,Luck_book}. Even for the case
$\pi= \BZ^2$ and $\cC$ is a $2$-term  complex $ 0 \to C_1 \to C_0 \to 0$, so that only $H_0(\cC)$ is non-trivial, there is still no proof of the conjecture that the $L^2$-torsion
is equal to either of the above limits. For results and discussions of this conjecture, see \cite{Luck_book,Luck,BV,Le_slides,Le_AHG,FJ,SW}.

(b) It should be noted that the exact calculation of the torsion part of the homology of finite coverings, even in the abelian case, is very difficult, see  \cite{HS,MM,Porti} for some partial results.
\end{remark}

\subsection{Refinement} Suppose $\pi$ is residually finite and the $L^2$-homology $H^{(2)}_k(\cC)=0$ for some $k$. For any normal subgroup $\Gamma \lhd \pi$ of finite index, the
homology group $H_k(\cC_\Gamma)$ is a finitely-generated abelian group. Because $H^{(2)}_k(\cC)=0$ one should expect that $H_k(\cC_\Gamma)$ is negligible. In fact,
a theorem of L\"uck \cite{Luck} (and Kazhdan for this case) says that
$$
\lim_{\la \Gamma  \ra \to \infty, |\pi:\Gamma | < \infty}\frac{ \rk_\BZ  H_k(\cC_\Gamma)  }{|\pi:\Gamma |} =0.
$$
This means  the free part $H_k(\cC_\Gamma)_{free}$ of $H_k(\cC_\Gamma)$ is small compared to the index. There is another measure of the free part $H_k(\cC_\Gamma)_{free}$, denoted by $R_k(\cC_\Gamma)$ and
called the regulator, or volume, see \cite{BV}  and Section  \ref{sec.reg}. Another expression of the fact that $H_k(\cC_\Gamma)_{free}$ is small compared to the index is expressed in the following statement, which complements
the result of Kazhdan-L\"uck.
\begin{Thm} \label{thm.01a}
Suppose  $\cC$ is a finitely generated based free $\BZ[\pi]$-complex with $\pi=\BZ^n$ and $H_k^{(2)}(\cC)=0$ for some index $k$. Then
\be
\lim_{\la \Gamma  \ra \to \infty, |\pi:\Gamma | < \infty}\frac{ \ln \vol( H_k(\cC_\Gamma)_{free})  }{|\pi:\Gamma |} =0.
\label{eq.007}
\ee
\end{Thm}

One conjectures that  \eqref{eq.007} holds for a large class of groups including hyperbolic groups.
\begin{remark} In \cite{Ri}, it is 
  proved that there is a sequence $\Gamma_n$, with $\la \Gamma_n \ra \to \infty$, such that the limit on the left hand side of \eqref{eq.007} along
$\Gamma_n$ is 0. Theorem \ref{thm.01a} implies that the limit is 0 for {\em any} such sequence.
\end{remark}

\subsection{On the proofs} For the proofs we use tools in commutative algebras and algebraic geometry. In particular, we make essential use of the theory of
torsion points in $\BQ$-algebraic set (a simple version of the Manin-Mumford principle). We hope that the methods and results can be adapted to the case of  elementary amenable groups.

\subsection{Acknowledgements} I would like to thank M. Baker, H. Dao,  W. L\"uck, A. Thom, and U. Zannier  for helpful disussions.
\subsection{Organization of the paper} In Section \ref{sec.volume} we recall the notions of geometric determinant and volume. We discuss the relation between homology and
Ray-Singer torsions in Section \ref{sec.reg}. An overview of the theory of torsion points in algebraic set is given in Section \ref{sec.torsion1}.
Section \ref{sec.tech} contains a crucial growth estimate which is needed in the proofs of the main theorems, given in Section \ref{sec.proofs}.

\def\fh{\mathfrak h}
\def\cE{\mathcal E}
\section{Geometric determinant, lattices and volume in  based Hermitian spaces}
\label{sec.volume}
In this section we recall the definition of geometric determinant and basic facts about volumes of lattices in based Hermitian spaces.
\subsection{Geometric determinant}
 For a linear map $f: V_1\to V_2$, where each $V_i$ is a finite-dimensional Hermitian space the {\em geometric determinant} $\det'(f)$ is the product of all non-zero
singular values of $f$. Recall that $x\in \BR$ is singular value of $f$ if $x \ge 0$ and $x^2$ is an eigenvalue of $f^* f$. By convention $\det'(f)=1$ if $f$ is the 0 map.
Thus we always have $\det'(f) >0$.

Since the maximal singular value of  $f$ is the norm $||f||$, we have
\be
 {\det}'(f) \le ||f|| ^{\dim V_2} \qquad \text{if $f$ is non-zero.}
 \label{eq.103s}
 \ee

 \begin{remark}
The geometric meaning of ${\det}' f$ is the following. The map $f$ restricts to a linear isomorphism $f'$ from  $\im(f^*)$ to $\im(f)$, each is a Hermitian space.
Then $\det' f= |\det (f')|$, where the ordinary determinant $\det(f')$ is calculated using orthonomal bases of the Hermitian spaces.
\end{remark}

\subsection{Based Hermitian space and volume}
Suppose $W$ is a finite-dimensional {\em based Hermitian space}, i.e. a $\BC$-vector space equipped with an Hermitian product $(.,.)$ and a  preferred  orthonormal basis. The $\BZ$-submodule $\Omega
\subset W$ spanned by the basis is called the {\em fundamental lattice}.

For a $\BZ$-submodule (also called a {\em lattice}) $\Lambda  \subset W$ with $\BZ$-basis $v_1,\dots,v_l$ define
$$ \vol(\Lambda ) = |\det \left ( (v_i,v_j)_{i,j=1}^l \right)|^{1/2}.$$
By convention, the volume of the 0 space is 1. If $\Lambda\subset \Omega$, we say that $\Lambda$ is an {\em integral lattice}.
It is clear that  $\vol(\Lambda ) \ge 1$ if $\Lambda$ is an integral lattice.

For a $\BC$-subspace $ V \subset W$  the lattice $V^{(\BZ)}:=V \cap \Omega$ is called the {\em $\BZ$-support of $V$}. We define
$$ \vol(V) := \vol(V^{(\BZ)}).$$
A lattice $\Lambda  \subset \Omega$ is {\em primitive} if is cut out from $\Omega$ by some subspace, i.e. $\Lambda  = V^{(\BZ)}$ for some subspace $V \subset W$. By definition, any primitive lattice is integral.

As usual, we say that a subspace $V\subset W$ is {\em is defined over  $\BQ$} if it is defined by some linear equations with rational coefficients (using the coordinates in the preferred base). It is easy to see that $V$ is defined over $\BQ$ if and only if it is spanned by its $\BZ$-support.

Suppose $V_1, V_2$ are subspaces of $W$ defined over $\BQ$, and $f: V_1 \to V_2$ is a $\BC$-linear map. We say that $f$ is {\em integral} if $f(V_1^{(\BZ)}) \subset V_2^{(\BZ)}$.

We summarize some well-known properties of volumes of lattices (see e.g. \cite{Bertrand}).

\begin{proposition} Suppose $V_1, V_2$ are subspaces of $W$ defined over $\BQ$ of a based Hermitian space $W$ and $f: V_1 \to V_2$ is an integral, non-zero  $\BC$-linear map. Then

\begin{align}
\vol(V_1+V_2) & \le \vol(V_1)\, \vol(V_2)  \label{eq.103a}\\
 \vol(\ker f)  \vol [f (V_1^{(\BZ)})] & = {\det}' (f) \, \vol(V_1).
\label{eq.103}
\end{align}

\label{pro.volume}
\end{proposition}
For a detailed discussion of \eqref{eq.103} and its generalizations to lattices in $\BZ[\BZ^n]$, see \cite{Ri}.

\section{Regulator, homology torsion, and Ray-Singer torsion} \label{sec.reg}
In this section we explain the relation between the homology torsion and the combinatorial Ray-Singer torsion. Proposition \ref{pro.01} of this section will be
used in the proof of main theorems.

Throughout this section we fix a finitely-generated  based free $\BZ$-complex
$\cE$
$$
0 \to E_m \overset {d_{m-1}}\longrightarrow E_{m-1} \overset {d_{m-1}}\longrightarrow E_{m-2} \dots \overset {d_2}\longrightarrow E_{1}\overset {d_1}\longrightarrow E_{0}\to 0.
$$
Define a Hermitian product on $E_k \otimes _\BZ \BC$ such that the preferred base is an orthonormal base. Now $E_k \otimes _\BZ \BC$   becomes a based Hermitian space.

We use the notations
$$ Z_k = \ker d_k, \quad B_k = \im d_{k+1}, \quad \overline B_k = (B_k \otimes_\BZ \BC) \cap E_k 
.$$

Let
$d_{k}^*: E_{k-1}\to  E_k$ be adjoint of $d_k$ and
$ D_k: E_k \to E_k$ be defined by $$D_k = d_k^* d_k + d_{k+1}d_{k+1}^*.$$

\subsection{Ray-Singer torsion and homology torsion}
Define the Ray-Singer torsion  and the homology torsion of $\cE$ by
\begin{align*}
\tau^{RS}(\cE)& = {\prod_k}^* {\det}' (d_k) \in \BR_+,\\
\tau ^H(\cE) &= \left({\prod_k}^* |\Tor_\BZ(H_k(\cE))|\right)^{-1}.
\end{align*}

\begin{remark}
The Ray-Singer torsion and the homology torsion can be defined through the classical Reidemeister torsion as follows.

Let $\tilde \fh_k$ be an orthonormal basis of $\ker(Dk) \otimes_ \BZ \BC= H_k(\cE\otimes _\BZ \BC)$.
With the bases $\{\tilde \fh_k\}$ of the homology of $\cE\otimes _\BZ \BC$, one can define the Reidemeister torsion 
$\tau^{R}(\cE\otimes _\BZ \BC, \{\tilde \fh_k\})$, defined up to signs (see e.g. \cite{Turaev}). It is not difficult to show that 
$$\tau^{RS}(\cE) = \left| \tau^{R}(\cE\otimes _\BZ \BC, \{\tilde\fh_k\})\right|.$$

 Both $\overline B_k$ and $Z_k$ are primitive lattices in $E_k$, and $\overline B_k \subset Z_k$.
There is a collection $\fh_k$  of elements of $Z_k\subset E_k$ which descend to a basis of the group $Z_k/\overline B_k$, the free part of $H_k(\cE)$.
Since $\fh_k$ is a basis of $H_k(\cE\otimes _\BZ\BC)$, there is defined the Reidemeister torsion $\tau^R(\cE\otimes _\BZ\BC,\{\fh_k\})$.
It is not difficult to prove the following generalization of the Milnor-Turaev formula \cite{Milnor,Turaev}.
$$\tau^H(\cE) =  |\tau^R(\cE\otimes _\BZ\BC,\{\fh_k\})| .$$
\end{remark}

\subsection{Regulators} By definition, $H_k(\cE)= Z_k /B_k$. The $\BZ$-torsion of $H_k(\cE)$ is $\overline B_k/B_k$, and the free part $H_k(\cE)_{free}$ is isomorphic to $Z_k/\overline B_k$.
For this reason, we define the volume  $\vol(H_k(\cE)_{free})$ to be
$$ 
R_k(\cE):= \frac{\vol(Z_k)}{\vol(\overline B_k)}.$$
Here we follow the notation of \cite{BV}, where $R_k$ is called the regulator.
\no{
Define
$$ R(\cE)= {\prod_k}^* R_k(\cE).$$
}
Using Identity \eqref{eq.103}, one can prove (see \cite[Formula 2.2.4]{BV})
\be
\tau^{RS}(\cE) = \tau^H(\cE)\, \left ( {\prod_k}^* R_k(\cE) \right).
\label{eq.e55}
\ee
We will use the following estimate of the regulator.
\begin{proposition} Let $\tilde R_k := \vol(\ker D_k)$.
For every $k$ one has
$$ \tilde R_k \ge R_k \ge \frac 1{\tilde R_k}.$$
\label{pro.01}
\end{proposition}
\begin{proof} Let $W$ be the orthogonal complement of $B_k \otimes _\BZ \BC$ in $Z_k \otimes _\BZ \BC$, and  $p:Z_k \otimes _\BZ \BC \to W$ be the orthogonal projection. Then
$$ R_k = \vol(p(Z_k)).$$

By Hodge theory (for finitely-generated $\BZ$-complex),
$$ \ker(D_k) = E_k \cap W = W^{(\BZ)}.$$
It follows that $\ker(D_k) \subset p(Z_k)$, and hence $\vol(p(Z_k) \le \vol(\ker(D_k))$, or
\be R_k \le \tilde R_k. \label{eq.90}
\ee

By \cite[Proposition 1(ii)]{Bertrand},
\be
 R_k = \frac{\vol(Z_k)}{\vol(\overline B_k)} = \frac{[(W \cap Z_k^*): D_k]}{\tilde R_k},
 \label{e1}
 \ee
where
$ Z_k^*$ is the $\BZ$-dual of $Z_k$ in $\Z_k \otimes_ \BZ \BC$ under the inner product. Note that $Z_k^*$ is also the orthogonal projection of $E_k$ on to $Z_k \otimes_\BZ \BC$.

Since the numerator of \eqref{e1} is $\ge 1$, we have $ R_k \ge 1/\tilde R_k$, which, together with \eqref{eq.90}, proves the proposition.
\end{proof}

\section{Abelian groups, algebraic subgroups of $(\BC^*)^n$, and torsion points }
\label{sec.torsion1}

 We review some facts about
representation theory of finite abelian groups in  subsection \ref{ss4} and the theory of torsion points on rational algebraic sets (a simple version of Manin-Mumford principle) in subsections \ref{sec.torsion} and \ref{sec.tor2}.

\subsection{Decomposition of the group ring of a finite abelian group}
\label{ss4}
Suppose $A$ is a finite  abelian group. The group ring $\BC[A]$ is an $A$-module (the regular representation) and is a  $\BC$-vector space of dimension $|A|$.
Equip $\BC[A]$ with a Hermitian product so that $A$ is an orthonormal basis. This makes $\BC[A]$ a based Hermitian space, with
 $\BZ[A]$  the fundamental lattice.

Let $\hat A = \Hom(A, \BC^*)$, known as the Pontryagin dual of $A$, be the group of all characters of $A$. Here $\BC^*$ is the multiplicative group of non-zero complex numbers.
We have $|\hat A|=|A|$.

The theory of representations of $A$ over $\BC$ is easy: $\BC[A]$ decomposes as a direct sum of mutually orthogonal one-dimensional $A$-modules:

\be \BC[A] = \bigoplus_{\chi \in \hat A} \BC e_\chi,
\label{e13}
\ee
where  $e_\chi$ is the idempotent
\be
 e_\chi = \frac{1}{|A|} \sum_{a \in A} \chi(a^{-1} ) a.
 \label{eq.10}
 \ee

The vector subspaces $\BC e_\chi$'s are not only orthogonal with respect to the Hermitian structure, but also orthogonal with respect to the ring structure in the sense that $e _\chi \, e_{\chi'} =0$ if $\chi \neq \chi'$. Each $\BC e_\chi$ is an ideal of the ring $\BC[A]$.

From the trace identity (see e.g. \cite[Section 2.4]{Serre}) we have, for every $a\in A$,
\be \sum_{\chi\in \hat A} \chi(a) = \begin{cases}
0 \quad &\text{if } a \neq e \\
|A| \quad & \text{if } a \neq e.
\end{cases}
\label{eq5}
\ee
Here $e\in A$ is the trivial element.

\def \cU {\mathbb U}
\def\Gal{\mathrm{Gal}}
\def\bu{\mathbf{u}}
\def\bv{\mathbf{v}}

\subsection{Algebraic subgroups of $(\BC^*)^n$ and lattices in $\BZ^n$} \label{sec.torsion}

\subsubsection{Algebraic subgroups of $(\BC^*)^n$}
An {\em algebraic subgroup of $(\BC^*)^n$} is a subgroup which is closed in the Zariski topology.

For a lattice $\Lambda$, i.e. a subgroup  $\Lambda$ of $\BZ^n$, not necessarily of maximal rank, let
$G(\Lambda)$  be the set of all  $\bz \in \BC^n $ such that $\bz^\bk=1$ for every $\bk \in \Lambda$. Here for $\bk=(k_1,\dots,k_n) \in \BZ^n$ and $\bz=(z_1,\dots,z_n)\in (\BC^*)^n$ we set
$\bz^\bk = \prod_i z_i^{k_i}$.

It is easy to see that $G(\Lambda)$ is an algebraic subgroup.  The converse holds true: Every algebraic subgroup is equal to $G(\Lambda)$ for some lattice $\Lambda$, see  \cite{Schmidt_a}.
If $\Lambda$ is primitive, then $G(\Lambda)$ is connected, and in this case it is called a {\em torus}. 

\subsubsection{Automorphisms of $(\BC^*)^n$}\label{sec.auto}
An example of a torus of dimension $l$ is the standard $l$-torus $T=(\BC^*)^l \times 1^{n-l} \subset (\BC^*)^n$, which is $G(\Xi_{n-l})$, where
$$\Xi_{n-l}= \{(k_1,\dots,k_n)\in \BZ^n  \mid
k_1=\dots=k_l=0\}.$$
The following trick shows that each torus is isomorphic to the standard torus. For details see \cite{Schmidt_a}.

For matrix $K \in GL_n(\BZ)$ with entries $(K_{ij})_{i,j=1}^n$, one can define an automorphism $\varphi_K$ of $(\BC^*)^n$ by
$$\varphi_K (z_1,z_2,\dots,z_n) = \left( \prod_{j=1}^n z_j^{K_1j},  \prod_{j=1}^n z_j^{K_2j}, \dots, \prod_{j=1}^n z_j^{K_nj} \right).$$
For any lattice $\Lambda\subset \BZ^n$, $\varphi_K(G(\Lambda))= G(K(\Lambda))$.
When $\Lambda$ is a primitive lattice of rank $n-l$, there is $K\in GL_n(\BZ)$ such that $K(\Lambda)=\Xi_{n-l}$. Then $\varphi_K(G(\Lambda))$ is the standard $l$-torus.

\subsubsection{Algebraic subgroups and character groups} Fix generators $t_1,\dots,t_n$ of $\BZ^n$. We will write $\BZ^n$ multiplicatively and use the identification
$\BZ[\BZ^n]= \BZ[t_1^{\pm1}, \dots, t_n^{\pm 1}]$.

Suppose $\Gamma\subset \BZ^n$ is a lattice.
Every element $\bz \in G(\Gamma)$ defines a character $\chi_\bz$ of the quotient group $A_\Gamma:= \BZ^n/\Gamma$ via
$$ \chi_\bz(t_1^{k_1} \dots t_n ^{k_n}) = \bz^\bk,   \quad \text{where } \bk=(k_1,\dots,k_n).$$
Conversely, every character of $A_\Gamma$ arises in this way.
Thus one can identify $G(\Gamma)$ with  the Pontryagin dual $\hat A_\Gamma$ via $\bz \to \chi_\bz$.

We will write $e_\bz$ for the idempotent $e_{\chi_\bz}$, and the decomposition \eqref{e13}, with $\Gamma$ having maximal rank, now becomes
\be \BC[A_\Gamma] = \bigoplus_{\bz \in G(\Gamma)} \BC \,e_\bz.
\label{e133}
\ee

\subsection{Torsion points in $\BQ$-algebraic sets}
\label{sec.tor2}

\subsubsection{Torsion points} With respect to the usual multiplication $\BC^* := \BC \setminus \{0\}$ is an abelian group, and so is the direct product $(\BC^*)^n$.
The subgroup of torsion elements of $\BC^*$, denoted by $\cU$,  is the group  of roots of unity, and $\cU^n$ is the torsion subgroup of $(\BC^*)^n$.

If $\Gamma \subset \BZ^n$ is a lattice of maximal rank, then $G(\Gamma)$ is finite, and $G(\Gamma)\subset \cU^n$.
\subsubsection{Torsion points and torsion coset} A {\em torsion coset} is any set of the form $\bz\,  G(\Lambda)$, where $\bz \in \cU^n$ and $\Lambda\subset \BZ^n$ is primitive, i.e. $G(\Lambda)$ is a torus.

Suppose $Z \subset \BC^n$. A torsion coset $X\subset Z$ is called a {\em maximal torsion coset in $Z$} if it is not a proper subset of any torsion coset in $Z$.

The following fact, known as  the Manin-Mumford theory for torsion points, is well-known, see   \cite{Laurent,Schmidt_a}.
\begin{proposition}\cite{Laurent} Suppose  $Z \subset \BC^n$ is an algebraic closed set  defined over $\BQ$. There are in total only  a finite number of maximal torsion cosets $X_j \subset Z$, $j=1, \dots, q$.
A torsion point $\bz\in \cU^n $ belongs to $Z$  if and only if
$z \in X_j$ for some $j$, i.e.
\be
Z \cap \cU^n = \bigcup_{j=1}^q (X_j \cap \cU^n).
\ee
\label{p.10}
\end{proposition}

\def\cl{\mathrm{cl}}
\subsubsection{$\BQ$-closure}Let $\Gal(\BC/\BQ)$ be the set of all field automorphisms of $\BC$ fixing every point in $\BQ$. For
$\bz =(z_1,\dots,z_n)$, a {\em Galois conjugate} of $\bz$ is any point of the form $\sigma(\bz)= (\sigma(z_1), \dots, \sigma(z_n))$, where $\sigma\in \Gal(\BC/\BQ)$.

Suppose $X\subset \BC^n$. Define the {\em $\BQ$-closure of $X$} by  
$$\cl_\BQ(X):=\cup_{\sigma \in \Gal(\BC/\BQ)} \sigma(X).$$

If $X\subset Z$, where $Z\subset \BC^n$ is an algebraic set
defined over $\BQ$, then $\cl_\BQ(X) \subset Z$.

If  $\bz \in \cU^n$ is a torsion point, then the $\BQ$-closure of $\{\bz\}$, also denoted by $\cl_\BQ(\bz)$, consists of a finite number of torsion points. One can prove that if
torsion order of $\bz$ is $k$, then $|\cl_\BQ(\bz)|= \phi(k)$, where $\phi$ is the Euler totient function. Though we don't need this result.

\begin{lemma} \label{lem.cl}
Let $X\subset (\BC^*)^n$ be a torsion coset. Then there exists a torsion point $\bu\in \cU^n$ and a primitive lattice $\Lambda\subset \BZ^n$ such that 
$$\cl_\BQ(X) = \bigsqcup_{\bz \in \cl_\BQ(\bu)} \bz\, G(\Lambda).$$
\end{lemma}
\begin{proof} By definition, there is a torus $T$ of dimension $l$ and
 a torsion point $\bu'$ such that $X= \bu' T$. Using an automorphism as described in Subsection \ref{sec.auto},
we can assume that $T$ is the standard $l$-torus, $T:= (\BC^*)^l \times 1^{n-l}$. Let $\bu\in \cU^n$ be the point whose first $l$-coordinates are 1, and the $j$-th coordinate is the same as that of $\bu'$ for
$j >l$. Alternatively, $\bu$ is the only intersection point of $X$ and $1^l \otimes (\BC^*)^{n-l}$.

Then $\bu \, T = \bu' \,T=X$. Any Galois conjugate of $\bu$ is in $1^l \otimes (\BC^*)^{n-l}$. If $\bz, \bz'$ are two distinct Galois conjugates of $\bu$, then $\bz ^{-1} \bz' \in 1^l \otimes (\BC^*)^{n-l}$, and hence $\bz ^{-1} \bz' \not \in G(\Lambda)$. It follows easily that $\cl_\BQ(X) = \cl_\BQ(\bu\, T) = \bigsqcup_{\bz \in \cl_\BQ(\bu)} \bz\, T$.
\end{proof}

\section{A growth rate estimate} \label{sec.tech}
In this section we prove Proposition \ref{pro.01a}, establishing a crucial growth estimate of volumes of subspaces depending on a torsion coset $X$ and a lattice $\Gamma \subset \BZ^n$ of maximal rank. 

\subsection{Settings and notations}
Throughout this section fix a torsion coset $X \subset (\BC^*)^n$, $X \neq (\BC^*)^n$,   and a $k \times l$ matrix $\DD$ with entries in $\BZ[\BZ^n]=\BZ[t_1^{\pm1},\dots, t_n^{\pm1}]$.

By right multiplication, we consider $\DD$ as a $\BC[\BZ^n]$-morphism
\be \DD: \BC[\BZ^n]^k \to \BC[\BZ^n]^l.
\label{eq.act}
\ee

 Suppose $\Gamma \le \BZ^n$ is a lattice of maximal rank. In what follows we fix $X, \DD$ but vary $\Gamma$.

Let $A= A_\Gamma := \BZ^n /\Gamma$, a finite abelian group. Equip $\BC[A]$ with the structure of a based Hermitian space as in Section \ref{ss4}. The fundamental lattice of $\BC[A]$ is $\BZ[A]$.

The map $\DD$ in \eqref{eq.act} descends to an integral $\BC[A]$-morphism
$$ \DD_\Gamma: \BC[A]^k \to \BC[A]^l.$$

By Lemma \ref{lem.cl}, there is a primitive lattice $\Lambda\le \BZ^n$ and  a
torsion point $\bu\in \cU^n$ such that
\be
 \cl_\BQ(X) = \bigsqcup_{j=1}^r \bu_j G(\Lambda),
 \label{eq.decomp}
 \ee
where $\{\bu_j, j=1\dots, r\}$ is the set of all Galois conjugates of $\bu$. Since $X \neq (\BC^*)^n$, $\Lambda$ is not the trivial group, $\Lambda \neq \{0\}$.

Let $B= B_{\Gamma,\Lambda}:= (\Gamma+ \Lambda)/\Gamma$. Then $B$ is a subgroup of $A=\BZ^n/\Gamma$.

\subsection{Decomposition of $\DD_\Gamma$ and norm of $\DD_\Gamma$}
Recall that $\hat A= G(\Gamma)$ and we have the decomposition \eqref{e133} of $A$-module
$$ \BC[A] = \bigoplus_{\bz \in \hat A} \BC e_\bz.$$
For each $\bz\in \hat A= G(\Gamma)$, $\DD$ induces a $\BC[A]$-map
$$ \DD(\bz): (\BC e_\bz)^k \to (\BC e_\bz)^l.$$
Here $\DD(\bz)$ is simply the $k\times l$ matrix with entries in $\BC$, obtained by evaluating $\DD$ at $\bz$. (Recall that each entry of $\DD$ is a Laurent polynomial in
$n$ variables, and one can evaluate such a Laurent polynomial at any point $\bz\in (\BC^*)^n$.)
We have
 $$\DD_\Gamma = \bigoplus_{\bz \in \hat A} \DD(\bz).$$

 It follows that
 $$ ||\DD_\Gamma|| = \max_{\bz \in \hat A} ||\DD(\bz)||.$$
  It is easy to see that the norm of any $k\times l$ matrix  is less than or equal to $kl$ times the sum of  the absolute values of all the entries.

 For a Laurent polynomial $f \in \BZ[\BZ^n]$ let the {\em $\ell^1$-norm of $f$}\  be the sum of the absolute values of all the coefficients of $f$.
 Let $||\DD||_1$ be the sum of the $\ell^1$-norms of all of its entries. Then we have $||\DD(\bz)|| \le kl ||\DD||_1$ because  each component of $\bz$ has absolute value 1. Thus we have the following uniform upper bound for $\DD_\Gamma$:
 \be ||\DD_\Gamma|| \le kl ||\DD||_1.
 \label{eq.301}
 \ee

\subsection{Integral decomposition of $\DD_\Gamma$ along $X$} \label{sec.Xj}
For each $\bz \in G(\Gamma)$, the 1-dimensional vector space $\BC e_\bz$ is in general not defined over $\BQ$. However, its $\BQ$-closure is defined over $\BQ$.

Consider the following $\BC[A]$-submodule    $\al(\Gamma,X)$ of $\BC[A]$:
$$\al(\Gamma,X):= \bigoplus_{\bz \in \hat A \cap \cl_\BQ(X)} \BC e_\bz \, \subset\,  \BC[A] = \bigoplus_{\bz \in \hat A} \BC e_\bz.$$
Since the set $\hat A \cap \cl_\BQ(X)$ is closed under Galois conjugations, $\al(\Gamma, X)$ is defined over $\BQ$. The orthogonal complement of $\al(\Gamma,X)$ is
$$\al^\perp(\Gamma,X):= \bigoplus_{\bz \in \hat A \setminus \cl_\BQ(X)} \BC e_\bz.$$

We have  $\BC[A]= \al(\Gamma,X) \oplus \al(\Gamma,X)$, and each of $\{ \al(\Gamma,X), \al^\perp(\Gamma,X)\}$ is an ideal of $\BC[A]$, or an $A$-subspace of $\BC[A]$.

The $A$-morphism $\DD_\Gamma$ restricts to  $A$-morphisms
$$ \DD_{\Gamma,X} : \al(\Gamma,X)^k \to \al(\Gamma,X)^l \quad \text{and} \quad  \DD_{\Gamma,X}^\perp : (\al^\perp(\Gamma,X))^k \to  (\al^\perp(\Gamma,X))^l, $$
which are integral.
We have $\DD_\Gamma = \DD_{\Gamma,X} \oplus \DD_{\Gamma,X}^{\perp}.$

\subsection{Projection onto $\al_{\Gamma,X}$}

\begin{lemma} \label{lem.a1}
If $G(\Gamma) \cap X \neq \emptyset$, then
\be\dim_\BC(\al_{\Gamma,X})= r|A|/ |B|
\label{eq.201a}
\ee
 and
 the orthogonal projection from $\BC[A]$ onto  $\al_{\Gamma,X}$ is given by the idempotent
 \be
 N_X:= \frac{1}{|B|}\sum_{j=1}^r \sum_{b \in B} \bu_j\,  (b^{-1}) b.
 \label{eq.11}
 \ee
\end{lemma}

\begin{proof}
Since $\hat A=G(\Lambda)$ is defined over $\BQ$, if it intersects $X$, then it intersects every component $\bu_j\,  G(\Lambda)$ of the decomposition \eqref{eq.decomp} of $\cl_\BQ(X)$.
 Let $\bu_j' \in G(\Gamma) \cap \bu_j \, G(\Lambda)$. Since $G(\Lambda)$ is a subgroup of $(\BC^*)^n$ and $\bu_j' \in \bu_j \, G(\Lambda)$, we have $\bu_j \, G(\Lambda)= \bu_j' \, G(\Lambda)$
 and
 \be \bu_j^{-1} \bu_j' \in G(\Lambda).
 \label{eq.b1}
 \ee

Since $G(\Gamma)$ is a subgroup and $\bu_j'\in G(\Gamma)$, we have $G(\Gamma)= \bu_j'\, G(\Gamma)$, and hence
$$G(\Gamma) \cap \bu_j \, G(\Lambda) = \bu_j' [ G(\Gamma) \cap G(\Lambda)]= \bu_j'[G(\Gamma+ \Lambda)].$$

From the above identity and the decomposition  \eqref{eq.decomp} we have   
\be
G(\Gamma) \cap \cl_\BQ(X) = \bigsqcup_{j=1}^r \bu_j'[G(\Gamma+\Lambda)].
\label{eq.6}
\ee

Let $s= |A|/|B|$, which is the cardinality of the quotient group $A/B=\BZ^n /(\Gamma+\Lambda)$. Then $|G(\Gamma+\Lambda)|= s$, since $G(\Gamma+\Lambda)$ is the Pontryagin dual of $A/B=\BZ^n /(\Gamma+\Lambda)$.
From \eqref{eq.6} we have  $|\hat A\cap \cl_\BQ(X) |= rs$. Hence $\dim_\BC(\al_{\Gamma,X})= rs=r|A|/ |B|$. This proves \eqref{eq.201a}.

For each element of the quotient group $A/B$ choose a lift in $A$, and denote by $C$ the set of all such lifts. We assume that the chosen lift of the trivial element is $e$.
Then
$$A =\{ bc \mid b\in B, c \in C\}.$$
From  \eqref{eq5}, for every $c\in C$,
\be
\sum_{\bz\in G(\Gamma+ \Lambda)} \bz(c) = \begin{cases}
0 \quad &\text{if } c \neq e \\
|C| \quad & \text{if } c= e.
\end{cases}
\label{eq.5a}
\ee

Recall that for $\bz \in \hat A= G(\Lambda)$, $e_\bz$ is an idempotent in $\BC[A]$, and $e_\bz e_{\bz'}=0$ if $\bz \neq \bz '$. It follows that the projection onto $\al_{\Gamma,X}$ is given by
$N_X =\sum_{\bz \in \hat A \cap \cl_\BQ(X) } \,  e_\bz$.

Using \eqref{eq.10}, we have
\begin{align*}
N_X  &=\frac 1 {|A|} \sum_{\bz \in \hat A \cap \cl_\BQ(X) } \,  \sum_{a\in A} \bz(a^{-1}) a 
\\
&= \frac 1 {|A|} \sum_{\bz \in \hat A \cap \cl_\BQ(X) }  \sum_{b\in B}  \sum_{c\in C} \bz(b^{-1}c^{-1}) b c=  \frac 1 {|B| |C|} \sum_{\bz \in \hat A \cap \cl_\BQ(X) }  \left[ \sum_{b\in B} \bz(b^{-1}) b \right] \left[  \sum_{c\in C} \bz(c^{-1}) c\right]\\
&= \frac 1 {|B| |C|} \sum_{j=1}^r\,  \sum_{\bz\in G(\Gamma+\Lambda)}\,   \left[ \sum_{b\in B} \bz(b^{-1}) \bu_j'(b^{-1})  b \right]    \left[ \sum_{c\in C} \bz(c^{-1}) \bu_j'(c^{-1})  c \right]  \quad \text{by \eqref{eq.6}}\\
&= \frac 1 {|B| |C|} \sum_{j=1}^r\,  \sum_{\bz\in G(\Gamma+\Lambda)}\,  \left[ \sum_{b\in B}  \bu_j'(b^{-1})  b \right]  \left[ \sum_{c\in C} \bz(c^{-1}) \bu_j'(c^{-1})  c\right]  \quad \text{because $ \bz(b^{-1})=1$}  \\
&= \frac 1 {|B| |C|} \sum_{j=1}^r\,   \left[ \sum_{b\in B}  \bu_j'(b^{-1})  b  \right]  \sum_{\bz\in G(\Gamma+\Lambda)}\,    \sum_{c\in C} \bz(c^{-1}) \bu_j'(c^{-1})  c  \\
&= \frac 1 {|B|} \sum_{j=1}^r\,   \sum_{b\in B}  \bu_j'(b^{-1}) \, b \qquad \text{by \eqref{eq.5a}} \\
&= \frac 1 {|B|} \sum_{j=1}^r\,   \sum_{b\in B}  \bu_j(b^{-1})\,  b \qquad \text{by \eqref{eq.b1}} .
\end{align*}
This completes the proof of the lemma.
\end{proof}
\subsection{Upper bound for the volume of $\al_{\Gamma,X} $}
\begin{lemma}

One has
 \be  \vol(\al_{\Gamma,X}) \le (r|B|)^{r |A|/|B|}.
 \label{eq.201}
 \ee
\end{lemma} \begin{proof}
We assume that $G(\Gamma)\cap X \neq \emptyset$, because otherwise $\al_{\Gamma,X}=0$, and the statement is trivial.

From \eqref{eq.11},
 $$|B| N_X = \sum_{b \in B }\left( \sum_{j=1}^r \bu_j(b^{-1}) \right)\, b \in \BZ[B].$$

 If  $a\in A$, then from \eqref{eq.11} we have
$$| B|N_X \, a =  \sum_{j=1}^r\,   \sum_{b\in B}  \bu_j(b^{-1})\,   b a.$$
 Because $|| \bu_j (b^{-1})\,  b a ||=1$, we see that $|B| N_X\, a$ has length  $\le r |B|$.

By Lemma \ref{lem.a1},  $\al(\Gamma,X)$ has dimension $rs$.
Since $A$ spans $\BC[A]$ and $N_X$ is the projection onto $\al_{\Gamma,X}$, the set $\{ N_X \, a \mid a \in A\}$ spans $\al_{\Gamma,X}$. Hence, there are $a_1,\dots,a_{rs} \in A$ such that
$\{ N_X a_j, j=1,\dots,rs\}$ is a basis of the vector space $\al_{\Gamma,X}$. Since $|B| N_X a_j \in \BZ[A]$ is integral, $\vol(\al_{\Gamma,X})$ is less than or equal to the volume of the parallelepiped spanned
by $\{|B| N_X a_j, j=1,\dots,rs\}$. The length of each $|B| N_X a_j$ is $\le r|B|$. It follows that $\vol(\al) \le (r|B|)^{rs}= (r|B|)^{r|A|/|B|}$.
\end{proof}
\subsection{Growth of $|B|$} The following statement has been used in \cite{Le_AHG}. This is the place we use the assumption $\la \Gamma \ra \to \infty$ and $\Lambda \neq 0$.
\begin{lemma} Suppose $\Lambda$ is a non-trivial lattice.
Then
\be \lim_{\la \Gamma \ra \to \infty} |B|=\infty.
\label{eq.400}
\ee
\end{lemma}
\begin{proof}  For the length of $x\in \pi=\BZ^n$ (in the definition of $\la \Gamma \ra$) we will use the standard metric $|x|$ derived from the Hermitian structure.

Fix an element $x \in \Lambda$, $x \neq 0$, and look at the degree of $x$ in $B= (\Lambda +\Gamma)/\Gamma$.
If $M |x| < \la \Gamma \ra$ for some positive integer $M$, then $M|x|$ does not belong to $\Gamma$ by the definition of $\la \Gamma \ra$, and hence $Mx$ is not 0 in $B= (\Lambda +\Gamma)/\Gamma$.
This means the cyclic subgroup of $B$ generated by $x$ has order at least $\la \Gamma \ra /|x|$.
It follows that $|B| \ge \la \Gamma \ra /|x|$.  Hence
$ \lim_{\la \Gamma \ra \to \infty } |B| = \infty$.
\end{proof}

\subsection{Growth of $\ \ker(\DD_{\Gamma,X})$}
\begin{proposition} For a fixed torsion coset $X\subset (\BC^*)^n$ and a matrix $\DD$ with entries in $\BZ[\BZ^n]$,  we have
$$ \lim_{\la \Gamma \ra \to \infty, |\BZ^n: \Gamma|< \infty} \frac {\ln \left[ \vol(\ker(\DD_{\Gamma,X}))\right]}{|\BZ^n: \Gamma| }=0.$$
\label{pro.01a}
\end{proposition}
\begin{proof} To simplify notations we will write $\al= \al_{\Gamma,X}$.  By \eqref{eq.103},

$$   \vol(\ker \DD_{\Gamma,X} ) \, \vol(\im \DD_{\Gamma,X} ^{(\BZ)}) = {\det}' \DD_{\Gamma,X} \, \vol(\al).$$
Because $\vol(\im \DD_{\Gamma,X} ^{(\BZ)}) \ge 1$, we have
\begin{align*}
\vol(\ker \DD_{\Gamma,X} ) & \le {\det}' \DD_{\Gamma,X} \, \vol(\al)  \\
&\le ||\DD_{\Gamma,X}|| ^ {\dim (\al)} \vol(\al)   \quad \text{by \eqref{eq.103s}}\\
&\le   ||\DD_{\Gamma,X} || ^ {r |A|/|B| }  (r |B|)^{r |A|/|B|} \quad \text{by \eqref{eq.201a} and \eqref{eq.201}}
\end{align*}
Since $\DD_{\Gamma,X}$ is a restriction of $\DD_\Gamma$ on a subspace, we have
$$ ||\DD_{\Gamma,X}|| \le ||\DD_{\Gamma}|| \le kl ||\DD||_1,$$
where the second inequality is \eqref{eq.301}.

It follows that
$$ \vol(\ker \DD_{\Gamma,X} ) \le (klr ||\DD||_1 |B|)^{r |A|/|B|},$$
and
\be
 \frac {\ln \left[ \vol(\ker(\DD_{\Gamma,X}))\right]}{|A| }\le  \frac {r \ln (klr ||\DD||_1 |B|)}{|B| }.
 \label{eq.202}
 \ee
Because $|B| \to \infty$ as $\la \Gamma \ra \to \infty$ by \eqref{eq.400}, the right hand side of \eqref{eq.202} goes to 0 as $\la \Gamma \ra \to \infty$.
\end{proof}

\section{Proofs of Theorems \ref{thm.01} and \ref{thm.01a}}\label{sec.proofs}
It is clear that Theorem \ref{thm.01} follows from Theorem \ref{thm.01a} and Identity \eqref{eq.e55}. We will prove Theorem \ref{thm.01a} in this section.

\subsection{Preliminaries}

Recall that $\cC$ is a finitely-generated based free $\BZ[\BZ^n]$-complex
$$
0 \to C_m \overset {\partial_m}\longrightarrow C_{m-1} \overset {\partial_{m-1}}\longrightarrow C_{m-2} \dots \overset {\partial_2}\longrightarrow C_{1}\overset {\partial_1}\longrightarrow C_{0}\to 0.
$$
Using the bases of $C_j$'s 
we identify $C_j$ with $\BZ[\BZ^n]^{b_j}$ and  $\partial_j$ with a $b_j \times b_{j-1}$ matrix with entries in $\BZ[\BZ^n]=\BZ[t_1^{\pm 1},\dots, t_n^{\pm 1}]$.

For $f=\sum a_j g_j\in \BZ[\BZ^n]$, where $a_j \in \BZ$ and $g_j \in \BZ^n$ let
$f^*= \sum a_j g_j^{-1}$. If $f$ acts on $\ell^2(\BZ^n)$, the Hilbert space with basis $\BZ^n$, by multiplication, then $f^*$ is  the adjoint operator of $f$.
As usual, for a matrix $O=(O_{ij})$ with entries in $\BZ[\BZ^n]$ let the {\em adjoint $O^*$} \ be defined by $(O^*)_{ij} = (O_{ji})^*$.

According to the assumption of Theorem \ref{thm.01a}, 
\be
H_k(\cC\otimes _\BZ[\BZ^n] F)=0, \label{eq.45}
\ee
 where $F$ is the fractional field of $\BZ[\BZ^n]$. 
Let 
$$\DD = \partial_k ^* \partial _k + \partial_{k+1} \partial _{k+1}^*: C_k \to C_k.$$
Then \eqref{eq.45}  is equivalent to the fact that $\DD$ is an injective map,
or that $\det(\DD)\neq 0$. Here $\det(\DD)$ is the usual determinant of a square  matrix with entries in a commutative ring. In our case $\det (\DD)$ is a Laurent polynomial in
$t_1,\dots, t_n$.

Let $Z$ be the zero set of the Laurent polynomial $\det(\DD)$ i.e. $Z =\{ \bz \in \BC^n \mid  \det(\DD)(\bz)=0\}$. In other words, $Z$ is the set of $\bz\in \BC^n$ such that the square matrix $\DD(\bz)$
is singular. Since $\det(\DD)\neq0$, $Z$ is not the whole $\BC^n$.

For every subgroup $\Gamma \le \BZ^n$, $\DD$ induces a map
$ \DD_\Gamma: \BC[A]^{b_k} \to \BC[A]^{b_k}$, where $A = \BZ^n/\Gamma$.

\subsection{Growth of $\vol(\ker(\DD_\Gamma))$}

\begin{proposition} In the above setting, we have
\label{pro.fi}
\be
\lim_{\la \Gamma  \ra \to \infty, |\pi:\Gamma | < \infty}\frac{\ln [\vol(\ker(\DD_\Gamma ))]}{|\pi:\Gamma |} =0.
\label{eq.7}
\ee
\end{proposition}
\begin{proof}  By Proposition \ref{p.10}, there are in total a finite number of maximal torsion cosets $X_j \subset Z$,
sets $j=1,\dots,q$, and
$$
Z  \cap \cU^n = \bigcup_{j=1}^q (X_j \cap \cU^n).
$$
Since $Z\neq \BC^n$,  none of the lattices associated to the torsion cosets $X_j$ is trivial.

Because $Z$ is defined over $\BQ$, any Galois conjugate of $X_j$ is also a maximal torsion coset in $Z$, i.e. among $\{X_1,\dots,X_q\}$.

Since $G(\Gamma) \subset \cU^n$, one has
$$ Z  \cap G(\Gamma)= (Z  \cap \cU^n)\cap G(\Gamma) =   \bigcup_{j=1}^q (X_j \cap \cU^n) \cap G(\Gamma)  =  \bigcup_{j=1}^q ( X_j \cap G(\Gamma)).
$$
Because the Galois conjugate of $X_j$'s are among the $X_j$'s, we also have

\be Z  \cap G(\Gamma)= \bigcup_{j=1}^q ( \cl_\BQ(X_j) \cap G(\Gamma)).
\label{eq.10d}
\ee
Let
$$ \al := \bigoplus_{\bz \in  G(\Gamma) \cap Z  } \BC e_\bz, \quad \al^\perp:= \bigoplus_{\bz \in G(\Gamma) \setminus Z } \BC e_\bz .$$
Then $\BC[A]= \al \oplus \al^\perp$, and each of $\{ \al, \al^\perp\}$ is an $A$-subspace of $\BC[A]$.
The linear operator $\DD_\Gamma$ restricts to
$$ \DD_{\Gamma,\al}: \al^{b_k} \to \al^{b_k} \quad \text{and} \quad \DD_{\Gamma,\al}^\perp : (\al^\perp)^{b_k} \to (\al^\perp)^{b_k},$$
and $$\DD_\Gamma = \DD_{\Gamma, \al} \oplus \DD_{\Gamma,\al}^\perp,$$
where
\be
\DD_{\Gamma, \al} = \bigoplus_{\bz \in  G(\Gamma) \cap Z  } \DD(\bz), \quad 
\DD_{\Gamma, \al}^\perp  = \bigoplus_{\bz \in  G(\Gamma) \setminus Z  } \DD(\bz).
\label{eq.65}
\ee
When $\bz \not \in Z$, $\DD(\bz)$ is non-singular. It follows that $\DD_{\Gamma,\al}^\perp$ is non-singular.
Hence,
\begin{align}
\ker(\DD_\Gamma) & = \ker(\DD_{\Gamma,\al})  \notag \\
& = \bigoplus_{\bz \in  G(\Gamma) \cap Z  } \ker(\DD(\bz)) \quad \text{by    \eqref{eq.65} }  \notag \\
&= \sum_{j=1}^q  \left (  \bigoplus_{\bz \in  \cl_\BQ(X_j) \cap G(\Gamma)   } \ker(\DD(\bz))    \right) \quad \text{by    \eqref{eq.10d} }  
\label{eq.60a}
\end{align}

From \eqref{eq.10d} we have
$$
\al = \sum_{j=1}^q   \left[ \bigoplus_{\bz \in  G(\Gamma) \cap \cl_\BQ(X_j ) }  \BC e_\bz \right]     =   \sum_{j=1}^q  \al(\Gamma,X_j),
$$
where $\al(\Gamma,X_j)$ is defined as in Section \ref{sec.Xj}. One also has
\be
 \ker(\DD_{\Gamma,X_j})= \bigoplus_{\bz \in  G(\Gamma) \cap \cl_\BQ(X_j ) } \ker(\DD(\bz)).
 \label{eq.60b}
 \ee
From \eqref{eq.60a} and \eqref{eq.60b}, one has

 $$ \ker (\DD_{\Gamma}) =\sum_{j=1}^q \ker(\DD_{\Gamma,X_j}).$$

Hence, by \eqref{eq.103a},
$$ \vol(\ker(\DD_\Gamma )) \le   \prod_{j=1}^q \vol(\ker(\DD_{\Gamma,X_j})).$$
Now  \eqref{eq.7} follows from the above inequality and  Proposition \ref{pro.01a}.
\end{proof}

\subsection{Proof of Theorem \ref{thm.01a}} Applying Proposition \ref{pro.01} to the $\BZ$-complex $\cE=\cC_\Gamma$, we get
$$ \vol(\ker(\DD_\Gamma)) \ge R_k(\cC_\Gamma)) \ge \frac{1}{\vol(\ker(\DD_\Gamma))}.$$

Thus we have

$$   \frac{\ln ( \vol(\ker(\DD_\Gamma)))}{ |\BZ^n : \Gamma|} \ge   \frac{\ln (R_k(\cC_\Gamma))}{ |\BZ^n : \Gamma|}\ge   -\frac{\ln ( \vol(\ker(\DD_\Gamma)))}{ |\BZ^n : \Gamma|}.$$
By Proposition \ref{pro.fi}, the limits of the two boundary terms, when $\la \Gamma \ra \to \infty$ and $ |\BZ^n: \Gamma| < \infty$,  are 0. Hence we also have
$$ \lim_{\la \Gamma \ra \to \infty, |\BZ^n: \Gamma| < \infty}\frac{\ln (R_k(\Gamma))}{ |\BZ^n : \Gamma|}=0.$$
This completes the proof of Theorem \ref{thm.01a}.

\end{document}